\documentclass[draft]{amsart}

\usepackage{amssymb}
\usepackage{url}
\usepackage{amscd}

\theoremstyle{plain}
\newtheorem{theorem}{Theorem}[section]

\newtheorem{corollary}[theorem]{Corollary}
\newtheorem{lemma}[theorem]{Lemma}
\newtheorem{proposition}[theorem]{Proposition}
\newtheorem{definition}[theorem]{Definition}
\theoremstyle{remark}

\newcommand{\K}{\mathcal{K}}
 \newcommand{\A}{\mathcal{A}}

\newcommand{\h}{\mathcal{H}}

\begin{document}


\title{Crossed Products by Partial actions of Inverse Semigroups}

\author{S. Moayeri\; Rahni}
\author{B. Tabatabaie\; Shourijeh}
\address{S. Moayeri\; Rahni, Department of Mathematics, College of Sciences, Shiraz University, Shiraz, 71454, Iran \\
              }
\email{smoayeri@shirazu.ac.ir}
\address{B. Tabatabaie\; Shourijeh, Department of Mathematics, College of Sciences, Shiraz University, Shiraz, 71454, Iran 
 }
\email{tabataba@math.susc.ac.ir}

\begin{abstract}
In this work, for a given inverse semigroup we will define the crossed product of an inverse semigroup by a partial action. Also, we will associate to an inverse semigroup $G$ an inverse semigroup $S_G$, and we will prove that there is a correspondence between the covariant representation of $G$ and covariant representation of $S_G$. Finally, we will explore a connection between crossed products of an inverse semigroup actions and crossed products by partial actions of inverse semigroups.
\end{abstract}

\maketitle


\newcommand\sfrac[2]{{#1/#2}}

\newcommand\cont{\operatorname{cont}}
\newcommand\diff{\operatorname{diff}}

\section{Introduction}
The theory of $C^*$-crossed product by group partial actions and inverse semigroup actions are very well developed \cite{partial} \cite{sieben}. In this paper, we show that the theory of crossed products by actions of inverse semigroups can be generalized to partial actions of inverse semigroups.\par In section \ref{two} we define a partial action of an inverse semigroup as a partial homomorphism from the inverse semigroup into a symmetric inverse semigroup on some set. We will refer the reader to \cite{EXPAN} for an extensive treatment of partial actions of inverse semigroups. In section \ref{two}, we define the crossed products by partial actions of inverse semigroups.\par It turns out that there is a close connection between crossed products by partial actions of inverse semigroups and crossed products by inverse semigroups actions. In section \ref{five}, we will show that every crossed products by partial action of an inverse semigroup is isomorphic to a crossed product by an inverse semigroup action.
\section{Partial Actions of Inverse Semigroups and Covariant Representations}\label{two}
We will assume that throughout this work $G$ is a unital inverse semigroup with unit element $e$ and $\A$ is a $C^*$-algebra.

We recall from ~\cite{EXPAN} that a partial action of an inverse semigroup $S$ on a set $X$ is a partial
homomorphism $\alpha :S\mapsto \verb"I"(X)$, that is, for each $s,t\in S$ $$\alpha(s^{*})\alpha(s)\alpha(t)=\alpha(s^{*})\alpha(st),\;\;\;\;\alpha(s)\alpha(t)\alpha(t^{*})=\alpha(st)\alpha(t^{*}),$$
where $\verb"I"(X)$ denotes the inverse semigroup of all partial bijections between subsets of $X$.
But we use ~\cite[Proposition 3.4]{EXPAN} to give a definition of a partial action of an inverse semigroup.
\begin{definition}\label{defpar}
Suppose that $S$ is an inverse semigroup and $X$ is a set. By a partial action of $S$ on
$X$ a map  we mean $\alpha :S\mapsto \verb"I"(X)$ satisfied the following conditions:
\begin{enumerate}
    \item [(\textit{i})]$\alpha_{s}^{-1}=\alpha_{s^{*}}$,
    \item
    [(\textit{ii})]$\alpha _{s}(X_{s^{*}}\cap X_t)=X_s\cap
    X_{st}$ for all $s,t\in S $ (where $X_s $ denotes the rang of $\alpha_{s}$ for each $s\in
    S$),
    \item [(\textit{iii})] $\alpha_s(\alpha_t(x))=\alpha_{st}(x)$ for all $x\in X_{t^*}\cap
    X_{t^*s^*}.$
\end{enumerate}
\end{definition}
To define a partial action $\alpha$ of an inverse semigroup $S$ on an
associative $\K$-algebra $\A$, we suppose in Definition \ref{defpar}
that each $X_s$ $(s\in S)$  is an ideal of $\A$ and that every map
$\alpha_s:X_{s^*}\mapsto X_s$ is an algebra isomorphism.
Furthermore, if the inverse semigroup $S$ is unital with unit $e$, we shall
suppose that $X_e=\A$. The next Proposition shows that for such a partial action $\alpha$ we have $\alpha_e$ is the identity map on $\A$. 
\begin{proposition}
If $\alpha$ is a partial action of $G$ on a $C^*$-algebra $\A$ then $\alpha_e$ is the identity map $\ell$ on $\A$.
\end{proposition}
\begin{proof}
By definition of partial action, $\alpha_e$ is an invertible map on it's domain, $D_e=\A$. Now, $$\ell=\alpha_e\alpha_e^{-1}=\alpha_e\alpha_{e^*}=\alpha_e\alpha_e=\alpha_e.$$
Note that we have used part (3) of Definition \ref{defpar} in the fourth equality above. 
\end{proof}
The following Lemma will be used in the proof of Theorem \ref{theorem1.4}
\begin{lemma}\label{lemma2.5}
If $\alpha$ is a partial action of $G$ on a $C^*$-algebra $\A$, then for all $t,s_1,...,s_n\in G$ $$\alpha_t(D_{t^*}D_{s_1}...D_{s_n})=D_{t}D_{ts_1}...D_{ts_n}.$$
\end{lemma}
\begin{proof}
For $t,s_1,...,s_n\in G$ we have
\begin{eqnarray*}
  \alpha_t(D_{t^*}D_{s_1}...D_{s_n}) &=& \alpha_t(D_{t^*}\cap D_{s_1}\cap ...\cap D_{t^*}\cap D_{s_n}) \\
   &=& \alpha_t(D_{t^*}\cap D_{s_1})\cap...\cap\alpha_t(D_{t^*}\cap D_{s_n}) \\
   &=& \alpha_t(D_t\cap D_{ts_1})\cap...\cap\alpha_t(D_t\cap D_{ts_n}) \\
   &=&  \alpha_t(D_t\cap D_{ts_1}\cap...\cap D_t\cap D_{ts_n})\\
   &=& \alpha_t(D_t D_{ts_1}... D_{ts_n})
\end{eqnarray*}

\end{proof}
\begin{theorem}\label{theorem1.4}
If $\alpha$ is a partial action of $G$ on a $C^*$-algebra $\A$, then for $s_1,...,s_n\in G$ the partial automorphism $\alpha_{s_1}...\alpha_{s_n}$ has domain\\ $D_{s_n^*}D{s_n^*s_{n-1}^*}...D_{s_n^*...s_1}$ and range $D_{s_1}...D_{s_1...s_n}$.
\end{theorem}
\begin{proof}
We will use induction to prove the statement about the domain. For $n=1$
\begin{equation*}
    dom \alpha_{s_1}=ran \alpha_{s^*}=D_{s^*}.
\end{equation*}
Now,
\begin{eqnarray*}
  dom \alpha_{s_1}...\alpha_{s_n} &=& \alpha_{s_n}^{-1}(dom (\alpha_{s_1}...\alpha_{s_{n-1}})\cap ran \alpha_{s_n}) \\
   &=& \alpha_{s_n^*}(D_{s_{n-1}^*}...D_{s_{n-1}^*...s_1^*}\cap D_{s_n}) \\
   &=& \alpha_{s_n^*}(D_{s_n}D_{s_{n-1}^*}...D_{s_{n-1}^*...s_1^*} ) \\
   &=& D_{s_n^*}D{s_n^*s_{n-1}^*}...D_{s_n^*...s_1}.
\end{eqnarray*}
Note that we obtained the last equality by using Lemma \ref{lemma2.5}. For the second statement, we have
\begin{eqnarray*}
  ran \alpha_{s_1}...\alpha_{s_n} &=& dom \alpha_{s_n^*}...\alpha_{s_1^*} \\
  &=& D_{s_1}...D_{s_1...s_n}
\end{eqnarray*}
by the first statement.
\end{proof}
If we consider a group $G$ as an inverse semigroup, then the two definitions of partial actions as a group and as an inverse semigroup are the same. This fact motivates us to define a covariant representation of a partial action of an inverse semigroup.
\begin{definition}\label{def1}
Let $\alpha$ be a partial action of $G$ on an algebra $\A$.  A covariant representation of $\alpha$
is a triple $(\pi, u, \h)$, where $\pi:\A\to B(\h)$ is a non-degenerate representation of $\A$ on a Hilbert space $\h$ and for each $g\in G$, $u_g$ is a partial isometry on $\h$ with initial space $\pi(D_{g^*})\h$ and final space $\pi(D_g)\h$, such that
\begin{enumerate}
  \item $u_g\pi(a)u_{g^*}=\pi(\alpha_g(a))\hspace{0.5cm}a\in D_{g^*}$,
  \item $u_{st}h=u_s u_th\hspace{0.5cm}$ for all $h\in \pi(D_{t^*}D_{t^*s^*})\h$,
  \item $u_{s^*}=u_s^*$.
\end{enumerate}
\end{definition}
Notice that by the \emph{Cohen-Hewitt factorization} Theorem $\pi(D_g)\h$ is a closed subspace of $\h$ and so the notions of initial and final spaces make sense.\par Now, we show that $u_e=1_{\h}$, where $e$ denotes the unit of $G$.
Since $D_e=\A$, by (2) of Definition \ref{def1} for all $h\in \pi(\A)\h=\h$ we have that $$u_eh=u_{ee}h=u_eu_eh.$$ Since $u_e$ is one to one on $\pi(\A)\h=\h$, we have $u_eh=h$ for all $h\in\h$ as we claimed.
\begin{definition}
Let $\alpha$ be a partial action of $G$ on a $C^*$-algebra $\A$. For $s\in G$, let $\rho_s$ denote the central projection of $\A^{**}$ which is the identity of $D_s^{**}$.
\end{definition}
Let ($\pi, u, \h$) be a covariant representation of ($\A, G, \alpha$). Since $\pi$ is a non-degenerate representation of $\A$, $\pi$ can be extended to a normal morphism of $\mathcal{A}^{**}$ onto $\pi(\A)''$. We will denote this extension also by $\pi$. Note that $\pi(D_{s_1}...D_{s_n})\h=\pi(\rho_{s_1}\ldots\rho_{s_n})\h$ for all $s_1,\ldots ,s_n\in G$, and $u_su_{s^*}=\pi(\rho_s)$ for all $s\in G$.
\begin{theorem}\label{theorem2.9}
Let ($\pi, u, \h$) be a covariant representation of ($\A, G, \alpha$). Then for all $s_1,...,s_n\in G$, $u_{s_1}...u_{s_n}$ is a partial isometry with initial space $$\pi(D_{s_n^*}D_{s_n^*s_{n-1}^*}...D_{s_n^*...s_1^*})\h$$ and final space $$\pi(D_{s_1}...D_{s_1...s_n})\h.$$
\end{theorem}
\begin{proof}
Firstly, we show that $u_{s_1}...u_{s_n}u_{s_n}^*...u_{s_1}^*=\pi(\rho_{s_1}...\rho_{s_1...s_n})$. For $n=1$ we have proved that $u_{s_1}u_{s_1}^*=\pi(\rho_{s_1})$. Now,
\begin{eqnarray*}
  u_{s_1}...u_{s_n}u_{s_n}^*...u_{s_1}^* &=& u_{s_1}u_{s_1}^*u_{s_1}\pi(\rho{s_2}...\rho_{s_2...s_n})ua_{s_1}^* \\
   &=& u_{s_1}u_{s_1^*}u_{s_1}\pi(\rho_{s_2}...\rho_{s_2...s_n})u_{s_1}^* \\
   &=& u_{s_1}\pi(\rho_{s_1^*})\pi(\rho_{s_2}...\rho_{s_2...s_n})u_{s_1}^*\\
   &=& u_{s_1}\pi(\rho_{s_1^*}\rho_{s_2}...\rho_{s_2...s_n})u_{s_1^*} \\
   &=& \pi(\alpha_{s_1}(\rho_{s_1^*}\rho_{s_2}...\rho_{s_2...s_n})) \\
   &=& \pi(\rho_{s_1}\rho_{s_1s_2}...\rho_{s_1s_2...s_n}),
\end{eqnarray*}
so, $u_{s_1}...u_{s_n}u_{s_n}^*...u_{s_1}^*$ is a projection since $\rho_{s_1},...,\rho_{s_1...s_n}$ are commute. Finally, the initial space of $u_{s_1}...u_{s_n}$ is equal to
\begin{eqnarray*}
u_{s_1}...u_{s_n}u_{s_n}^*...u_{s_1}^*\h&=&\pi(\rho_{s_1}...\rho_{s_1...s_n})\h  \\
  &=&\pi(D_{s_n^*}D_{s_n^*s_{n-1}^*}...D_{s_n^*...s_1^*})\h.
   \end{eqnarray*}
   Similarly, we can prove that the final space of $u_{s_1}...u_{s_n}$ is equal to $$\pi(D_{s_1}...D_{s_1...s_n})\h.$$
\end{proof}
\begin{corollary}\label{cor2.10}
If ($\pi, u, \h $) is a covariant representation of ($\A, G, \alpha$), then 
\begin{equation*}
u_{s_1...s_n}h=u_{s_1}...u_{s_{n}}h\;\text{ for all}\; h\in\pi(D_{s_n^*}...D_{s_n^*...s_1^*})\h,
\end{equation*}
and
\begin{equation*}
  \pi(a)u_{s_1...s_n}=\pi(a)u_{s_1}...u_{s_n}\;\text{for all}\; a\in D_{s_1}D_{s_1s_2}...D_{s_1...s_n}.
\end{equation*}
\end{corollary}
\begin{proof}
For $n=2$, if $h\in\pi(D_{s_2^*}D_{s_2^*s_1^*})$ then by Definition \ref{def1} part (2) we have $u_{s_1s_2}h=u_{s_1}u_{s_2}$.
For $h\in \pi(D_{s_n^*}...D_{s_n^*...s_1^*})\h$ we have
$u_{s_1...s_n}h=u_{s_1...s_{n-1}}u_{s_n}h$ by Definition \ref{def1} part (2). Now, since 
\begin{eqnarray*}
  u_{s_n}h\in u_{s_n}\pi(\rho_{s_{n}^*}...\rho_{s_n^*...s_1^*})\h &=& \pi(\rho_{s_n^*}\rho_{s_n^*s_{n-1}^*}...\rho_{s_{n-1}^*...s_1^*})\h \\
   &\subseteq& \pi(\rho_{s_{n-1}^*}...\rho_{s_{n-1}^*...s_1^*})\h \\
   &=& \pi(D_{s_{n-1}^*}...D_{s_{n-1}^*...s_1^*})\h,
\end{eqnarray*}
by induction hypothesis we have $u_{s_1...s_{n-1}}u_{s_n}h=u_{s_1}...u_{s_{n-1}}u_{s_n}h$. By the first statement, we have
\begin{equation}\label{eq1}
  u_{s_n^*s_{n-1}^*...s_1^*}\pi(a^*)=u_{s_n^*}...u_{s_1^*}\pi(a^*)
\end{equation}
since $\pi(a^*)\in\pi(D_{s_1}...D_{s_1...s_n})$ for $a\in D_{s_1}...D_{s_1...s_n}$. Taking the conjugate, we have $\pi(a)u_{s_1...s_n}=\pi(a)u_{s_1}...u_{s_n}.$
\end{proof}
\begin{corollary}\label{cor2.11}
If $(\pi, u, \h)$ is a covariant representation of ($\A, G, \alpha$), then $S=\{ u_{s_1}...u_{s_n}\;:\; s_1,..., s_n\in G\}$ is a unital inverse semigroup of partial isometries of $\h$.
\end{corollary}
Now, we are able to define an inverse semigroup associated to a covariant representation of a unital inverse semigroup $G$.
\begin{proposition}\label{pro3.3}
Let $\alpha$ be a partial action of an inverse semigroup $G$ on the $C^*$-algebra $\A$, and let ($\pi, u. \h$) be a covariant representation of $\alpha$. Let $S_G=\{(\alpha_{g_1}...\alpha_{g_n}, u_{g_1}...u_{g_n})\;:\;g_1,...,g_n\in G\}$. Then $S_G$ is a unital inverse semigroup with coordinate wise multiplication. For $s=(\alpha_{g_1}...\alpha_{g_n}, u_{g_1}...u_{g_n})\in S_G$ let
\begin{equation*}
  E_s=D_{g_1}...D_{g_1...g_n},
\end{equation*}
and
\begin{equation*}
  \beta_s=\alpha_{g_1}...\alpha_{g_n}:E_{s^*}\to E_s.
\end{equation*}
Then $\beta$ is an action of $S_G$ on $\A$.
\end{proposition}
\begin{proof}
Let $s=(\alpha_{g_1}...\alpha_{g_n}, u_{g_1}...u_{g_n}), t=(\alpha_{h_1}...\alpha_{h_n}, u_{h_1}...u_{h_n})$, then $st=(\alpha_{g_1}...\alpha_{g_n}\alpha_{h_1}...\alpha_{h_n}, u_{g_1}...u_{g_n}u_{h_1}...u_{h_n})\in S_G$, and the unit of $S_G$ is $(\alpha_e, u_e)$. Obviously $E_s$ is a closed ideal of $\A$ and $\beta_s$ is an isomorphism. Now, we define the domain of $\beta_s$.
\begin{eqnarray*}
  dom(\alpha_{g_1}...\alpha_{g_n}) &=& \alpha_{g_n^*}(dom(\alpha_{g_1}...\alpha_{g_{n-1}}) D_{g_n}) \\
   &=& \alpha_{g_n^*}(\alpha_{g_{n-1}^*}(dom(\alpha_{g_1}...\alpha_{g_{n-2}})D_{g_{n-1}}) D_{g_n})) \\
   &\vdots&\\
  &=& D_{g_n^*}...D_{g_{n}^*...g_1^*}=E_{s^*}.
\end{eqnarray*}
Now, let us show that $ran\beta_s=E_s$. To do this, we will use induction. For $n=2$,
\begin{eqnarray*}
 ran\beta_s&=&ran \alpha_{g_1}\alpha_{g_2} \\
   &=& \alpha_{g_1}(D_{g_2}D_{g_1^*}) \\
    &=& D_{g_1}D_{g_1g_2}=E_s.
\end{eqnarray*}
On the other hand, for $s=(\alpha_{g_1}...\alpha_{g_n}, u_{g_1}...u_{g_n})$,
\begin{eqnarray*}
  ran\beta_s &=& ran\alpha_{g_1}...\alpha_{g_n} \\
   &=& \alpha_{g_1}(ran(\alpha_{g_2}...\alpha_{g_n})D_{g_{1}^*}) \\
   &=& \alpha_{g_1}(D_{g_2}D_{g_2g_3}...D_{g_2...g_n}D_{g_{1}^*})\\
   &=& D_{g_1}D_{g_1g_2}...D_{g_1...g_n}=E_s.
\end{eqnarray*}
So, $\beta_s:E_{s^*}\to E_s$ is an isomorphism, and clearly for $s,t\in S$ we have $\beta_s\beta_t=\beta_{st}$.
\end{proof}
The following Proposition shows that there exists a relation between covariant representation of $(\A, S_G, \beta)$ and covariant representation of $(\A, G, \alpha).$
\begin{proposition}\label{pro3.5}
keeping the notation of Proposition \ref{pro3.3}, define $\nu:S_G\to B(\h)$ by $\nu_s=u_{g_1}...u_{g_n}$, where
$s=(\alpha_{g_1}...\alpha_{g_n}, u_{g_1}...u_{g_n})$. Then $(\pi, \nu, \h)$ is a covariant representation of $(\A, S_G, \beta)$. Conversely, if ($\rho, z, \K $) is a covariant representation of $(\A , S_G , \beta)$, then the function $\omega : G\to B(\K)$ defined by $\omega_g=z(\alpha_g,u_g)$ gives a covariant representation $(\rho, \omega, \K)$ of $(\A, G, \alpha)$.
\end{proposition}
\begin{proof}
Let $s=(\alpha_{g_1}...\alpha_{g_n}, u_{g_1}...u_{g_n})\in S$. By Theorem \ref{theorem2.9}, $\nu_s=u_{g_1}...u_{g_n}$ is a partial isometry with initial space $\pi(E_{s^*})\h$ and final space $\pi(E_s)\h$. Obviously, $\nu$ is multiplicative. Let $a\in E_{s^*}= D_{g_n^*}...D_{g_n^*...g_1^*}$, then
\begin{eqnarray*}
  \nu_s\pi(a)\nu_{s^*} &=& u_{g_1}...u_{g_n}\pi(a)u_{g_n^*}...u_{g_1^*} \\
  &=& u_{g_1}...u_{g_{n-1}}\pi(\alpha_{g_n}(a))u_{g_{n-1}^*}...u_{g_1^*}  \\
  &\vdots&\\
   &=& \pi(\alpha_{g_1}...\alpha_{g_n}(a))=\pi(\beta_s(a)).
\end{eqnarray*}
Conversely, suppose that ($\rho, z, \K$) is a covariant representation of $(\A, S_G, \beta)$. We want to show that ($\rho, \omega, \K$) is a covariant representation of ($\A, G, \alpha$). By the definition of $\omega_g$, $g\in G$, $\omega_g$ is a partial isometry with initial space $\pi(D_{g^*})\K$ and final space $\pi(D_g)\K$. For $g_1, g_2\in G$, put $s=(\alpha_{g_1g_2}, u_{g_1g_2})$, $s_1=(\alpha_{g_1}, u_{g_1})$, and
$s_2=(\alpha_{g_2}, u_{g_2})$. By the definition of partial action, if $x\in D_{g_2^*}D_{g_1^*}$ then $\alpha_{g_1}\alpha_{g_2}(x)=\alpha_{g_1g_2}(x)$. But,
\begin{equation*}
  ran\alpha_{g_2^*}\alpha_{g_1^*}=\alpha_{g_2^*}(D_{g_1^*}D_{g_2})=D_{g_2^*}D_{g_2^*g_1^*}.
\end{equation*}
Consequently,
\begin{equation}\label{eq2}
  \alpha_{g_1g_2}(\alpha_{g_1}\alpha_{g_2})^*=\alpha_{g_1g_2}\alpha_{g_2^*}\alpha_{g_1^*}=\alpha_{g_1}\alpha_{g_2}(\alpha_{g_1}\alpha_{g_2}).
\end{equation}
By Definition \ref{def1} part (2), $u_{g_1g_2}=u_{g_1}u_{g_2}$ on $\pi (D_{g_2^*}D_{g_2^*g_1^*})\h$. On the other hand, by Theorem \ref{theorem2.9} final space of $(u_{g_1}u_{g_2})^*$ is $\pi(D_{g_2^*}D_{g_2^*g_1^*})\h$. Thus
\begin{equation}\label{eq3}
  u_{g_1g_2}(u_{g_1}u_{g_2})^*=u_{g_1}u_{g_2}(u_{g_1}u_{g_2})^*.
\end{equation}
Hence
\begin{eqnarray}\label{eq4}
\nonumber  s(s_1s_2)^* &=& (\alpha_{g_1g_2}, u_{g_1g_2})[(\alpha_{g_1}, u_{g_1})(\alpha_{g_{2}}, u_{g_2}))]^* \\
 \nonumber  &=& (\alpha_{g_1}\alpha_{g_2}(\alpha_{g_1}\alpha_{g_2})^*, u_{g_1}u_{g_2}(u_{g_1}u_{g_2})^*) \\
  &=& s_1s_2(s_1s_2)^*,
\end{eqnarray}
note that we have used equations \ref{eq2} and \ref{eq3} in the second equality above. So,
\begin{eqnarray}\label{eq5}
  \nonumber z_sz_{(s_1s_2)^*}&=& z_{s(s_1s_2)^*} \\
  \nonumber &=& z_{s_1s_2(s_1s_2)^*} \\
  &=& z_{s_1s_2}z_{(s_1s_2)^*}
\end{eqnarray}
by equality \ref{eq4}. Now, for $h\in\rho(D_{g_2^*}D_{g_2^*g_1^*})\K$
\begin{eqnarray*}
  z_s h&=&z_{s_1s_2}h  \\
   &=& z_{s_1}z_{s_2}h
\end{eqnarray*}
note that we have used \ref{eq5} and the fact that $\rho(D_{g_2^*}D_{g_2^*g_1^*})\K$ is the final space of $z_{(s_1s_2)^*}$ in the first equality. Hence $\omega_{g_1g_2}=\omega_{g_1}\omega_{g_2}$ on $\rho(D_{g_2^*}D_{g_2^*g_1^*})\K$. For $g\in G$,
\begin{equation*}
  \omega_{g^*}=z_{(\alpha_{g^*}, u_{g^*})}=z_{s^*}=z^*_s=\omega^*_g.
\end{equation*}
Consequently, $(\A, G, \omega)$ is a covariant representation of $(\A, G, \alpha)$.
\end{proof}
\section{Crossed Products}\label{sec3}
Mc Calanahan defines the partial crossed product $\A\ltimes _{\alpha}G$ of the $C^*$-algebra $\A$ and the group $G$ by the partial action $\alpha$ as the enveloping $C^*$-algebra of $L=\{x\in \ell^1(G,\A)\;:\;x(g)\in D_g\}$ with the multiplication and involution
\begin{equation*}
  (x\ast y)(g)=\sum_{h\in G}\alpha_h[\alpha_{h^{-1}}(x(h))y(h^{-1}g)],
\end{equation*}
and
\begin{equation*}
  x^*(g)=\alpha_g(x(g^{-1})^*).
\end{equation*}
He shows that there is a bijective correspondence $(\pi, u, \h)\leftrightarrow (\pi\times u, \h)$ between covariant representations of $(\A, G, \alpha)$ and non-degenerate representations of $\A\ltimes_{\alpha} G$, where $\pi\times u$ defined by $x\mapsto\sum_{g\in G}\pi(x(g))u_g$. We are going to follow his footsteps constructing the crossed product of a $C^*$-algebra and a unital inverse semigroup by a partial action.\par Let $\alpha$ be a partial action of the unital inverse semigroup $G$ on the $C^*$-algebra $\A$. Consider the subset $L=\{x\in\ell^1(G,\A)\;:\;x(g)\in E_g\}$ of $\ell^1(G,\A)$ with the multiplication and involution as follows:
\begin{eqnarray*}
  (x\ast y)(g) &=& \sum_{hk=g}\beta_{h}(\beta_{h^*}(x(h))y(k)), \\
  x^*(g) &=& \beta_s(x(g^*)^*).
\end{eqnarray*}
Notice that by Definition \ref{defpar} part (ii), $(x\ast y)(g)\in E_g$. It is easy to see that for $x,y\in L$ we have $x\ast y, x^*\in L$, and
\begin{equation*}
  \| x\ast y \|\leq \|x\|\|y\|,
\end{equation*}
and
\begin{equation*}
  \|x^*\|=\|x\|.
\end{equation*}
Obviously, $L$ is a closed subset of $\ell^1(G, \A)$, so, $L$ is a Banach space. Easily one can shows that
\begin{itemize}
  \item [(\textit{i})] $(x+y)^*=x^*+ y^*,$
  \item [(\textit{ii})] $(ax)^*=\bar{a}x^*,$
  \item [(\textit{iii})] $(x\ast y)^*=y^*\ast x^*.$
\end{itemize}
\begin{proposition}\label{pro4.1}
$L$ is a Banach *-algebra.
\end{proposition}
\begin{proof}
By the argument above, $L$ is a Banach space closed under multiplication and involution. To show that $L$ is a Banach *-algebra, it is enough check the associativity of multiplication. It suffices to show this for $x=a_r\delta_r, y=a_s\delta_s,$ and $a_t\delta_t$. Let $\{u_{\lambda}\}$ be an approximate identity for $E_{s^*}$, then
\begin{eqnarray*}
  (a_r\delta_r\ast a_s\delta_s)\ast a_t\delta_t &=& \beta_r(\beta_{r^*}(a_r)a_s)\delta_{rs}\ast a_t\delta_t \\
  &=& \beta_{rs}(\beta_{s^*r^*}(\beta_r(\beta_{r^*}(a_r)a_s))a_t)\delta_{rst} \\
   &=&  \beta_{rs}(\beta_{s^*}\beta_{r^*}(\beta_r(\beta_{r^*}(a_r)a_s))a_t)\delta_{rst} \\
   &=& \beta_{rs}(\beta_{s^*}(\beta_{r^*}(a_r)a_s)a_t)\delta_{rst} \\
   &=&  \lim_{\lambda}\beta_{rs}(\beta_{s^*}(\beta_{r^*}(a_r)a_s)u_{\lambda}a_t)\delta_{rst}\\
   &=&  \lim_{\lambda}\beta_{r}\beta{s}(\beta_{s^*}(\beta_{r^*}(a_r)a_s)a_t)\delta_{rst}\\
   &=&\lim_{\lambda}\beta_r(\beta_{r^*}(a_r)a_s\beta_{s}(u_{\lambda}a_t))\delta_{rst}\\
   &=&\lim_{\lambda}\beta_r(\beta_{r^*}(a_r)\beta_{s}(\beta_{s^*}(a_s)u_{\lambda}a_t))\delta_{rst}\\
   &=&\beta_r(\beta_{r^*}(a_r)\beta_{s}(\beta_{s^*}(a_s)a_t))\delta_{rst}\\
   &=& a_r\delta_r\ast(\beta_{s}(\beta_{s^*}(a_s)a_t)))\delta_{st}\\
   &=& a_r\delta_r\ast(a_s\delta_s\ast a_t\delta_t).
\end{eqnarray*}
\end{proof}
Note that authors in \cite{skew} prove the associativity of $L$ in a general case, where $\A$ is just an algebra.
\begin{definition}\label{def4.2}
If $(\pi, \nu. \h)$ is a covariant representation of ($\A, S, \beta$), then define
  $\pi\times\nu: L\mapsto B(\h)$ by $(\pi\times\nu)(x)=\sum_{s\in S}\pi(x(s))\nu_s.$
\end{definition}
\begin{proposition}\label{pro4.3}
$\pi\times\nu$ is a non-degenerate representation of $L$.
\end{proposition}
\begin{proof}
clearly $\pi\times\nu$ is a linear map from $L$ into $B(\h)$. As for multiplicativity, it suffices to verify this for elements of the form $a_s\delta_s$. For such elements, we have
\begin{eqnarray*}
  \pi\times\nu(a_s\delta_s\ast a_t\delta_t) &=& \pi\times\nu(\beta_s(\beta_{s^*}(a_s)a_t)\delta_{st}) \\
  &=& \pi(\beta_s(\beta_{s^*}(a_s)a_t)\nu_{st}.
\end{eqnarray*}
We also have
\begin{eqnarray*}
  \pi\times\nu(a_s\delta_s)\pi\times\nu(a_t\delta_t) &=& \pi(a_s)\nu_s\pi(a_t)\nu_t \\
   &=& \nu_s\nu_{s^*}\pi(a_s)\nu_s\pi(a_t)\nu_t \\
  &=& \nu_s\pi(\beta_{s^*}(a_s))\pi(a_t)\nu_t \\
   &=& \nu_s\pi(\beta_{s^*}(a_s)a_t)\nu_t \\
   &=& \nu_s\pi(\beta_{s^*}(a_s)a_t)\nu_{s^*}\nu_s\nu_t \\
   &=& \pi(\beta_s(\beta_{s^*}(a_s)a_t))\nu_s\nu_t\\
\end{eqnarray*}
We have used the fact that $\nu_s\nu_{s^*}\pi(a_s)=\pi(a_s)\nu_{s^*}\nu_s=\pi(a_s)$ for $a\in E_s$ in the second and fifth equalities above. Since $\beta_s(\beta_{s^*}(a_s)a_t)$ is in $\beta_s(E_{s^*}E_{t})=E_sE_{st}$, it follows from Definition \ref{def1} that 
\begin{equation*}
  \pi(\beta_s(\beta_{s^*}(a_s)a_t))\nu_s\nu_t=\pi(\beta_s(\beta_{s^*}(a_s)a_t))\nu_{st},
\end{equation*}
 so, the multiplicativity of $\pi\times\nu$ follows. The following computations verify that $\pi\times\nu$ preserves the $*$-operation.
\begin{eqnarray*}
  \pi\times\nu((a_s\delta_s)^*) &=& \pi\times\nu(\beta_{s^*}(a_s^*)\delta_{s^*}) \\
  &=& \pi(\beta_{s^*}(a_s^*))\nu_{s^*} \\
   &=& \nu_{s^*}\pi(a_{s}^*)\nu_s\nu_{s^*} \\
   &=& \nu_{s^*}\pi(a_s^*) \\
   &=& (\pi(a_s)\nu_s)^*=(\pi\times\nu(a_s\delta_s))^*.
\end{eqnarray*}
If $\{u_{\lambda}\}$ is a bounded approximate identity for $\A$, then $\{u_{\lambda}\delta_e\}$ is a bounded approximate identity for $L$ since for $a\in E_s$ we have
\begin{equation*}
  \lim_{\lambda}u_{\lambda}\delta_e\ast a\delta_s=\lim_{\lambda}u_{\lambda}a\delta_s=a\delta_s,
\end{equation*}
and
\begin{equation*}
  \lim_{\lambda}a\delta_s\ast u_{\lambda}\delta_e=\lim_{\lambda}\beta_s(\beta_{s^*}(a)u_{\lambda})\delta_s=a\delta_s.
\end{equation*}
Since $\pi$ is a non-degenerate representation, $\pi\times\nu(u_{\lambda}\delta_e)=\pi(u_{\lambda})$ converges strongly to $1_{B(\h)}$ and so $\pi\times\nu$ is non-degenerate.
\end{proof}
\begin{definition}
Let $\A$ be a $C^*$-algebra and $\beta$ be a partial action of the unital inverse semigroup $G$ on $\A$. Define a seminorm $\|.\|_1$ on $L$ by
\begin{equation*}
  \|x\|_1=\sup\{\|\pi\times\nu(x)\|\;:\;(\pi,\nu)\textsl{is a covariant representation of } (\A, G, \beta)\},
\end{equation*}
and let $N=\{x\in L\;:\;\|x\|_1=0\}$
\end{definition}
The crossed product $\A\ltimes_{\beta} G$ is the $C^*$-algebra obtained by completing the quotient $\frac{L}{N}$ with respect to $\|.\|_1.$
\begin{lemma}\label{lem4.5}
If $s\leq t$ in $G$, then $\Phi(a\delta_s)=\Phi(a\delta_t)$ for all $a\in E_s$, where $\Phi$ is the quotient map of $L$ onto $\frac{L}{N}$.
\end{lemma}
\begin{proof}
Notice that since $s\leq t$  there is an idempotent $f$ in $G$ such that $s=ft$, and we have $E_s\subseteq E_t$ by \cite[Proposition 3.8]{EXPAN}, so, $a\in E_t$. If ($\pi, \nu$) is a covariant representation of ($\A, G, \beta$), then
\begin{eqnarray*}
  \pi\times\nu(a\delta_s- a\delta_t) &=& \pi(a)\nu_s - \pi(a)\nu_t \\
   &=& \pi(a)\nu_{ft} - \pi(a)\nu_t \\
   &=& \pi(a)\nu_f\nu_t- \pi(a)\nu_t.
\end{eqnarray*}
We have used the fact that for $a\in E_s$ $\pi(a)\nu_{ft}=\pi(a)\nu_f\nu_t$ in the the third equality. Since $f$ is an idempotent, $\nu_f$ is identity on $\pi(E_f)\h$. Now for $h\in \h$ if $\nu_t(h)\in \pi(E_f)\h$, then
\begin{equation*}
  \pi(a)\nu_f\nu_t(h) - \pi(a)\nu_t(h)=0.
\end{equation*}
If $\nu_t(h)\in (\pi(E_f)\h)^\bot= Ker\hspace{0.1cm} \nu_f$, then
\begin{equation*}
  \pi(a)\nu_f\nu_t(h)=0.
\end{equation*}
On the other hand, $\pi(a)\nu_t(h)=0$ because if $k\in\h$ then
\begin{equation*}
  <\pi(a)\nu_t(h), k>= <\nu_t(h), \pi(a^*)k>=0
\end{equation*}
since $a^*\in E_s=E_{ft}\subseteq E_f$. Hence, $\Phi(a\delta_s - a\delta_t)=0$. Note that the fact that $E_{ft}\subseteq E_f$ follows from \cite[Corollary 2.21]{EXPAN} and the fact that
\begin{equation*}
  E_{ft}= ran \beta_{ft}=\beta_f\beta_t=\beta_f(E_tE_f)=E_{ft}\cap E_f.
\end{equation*}

\end{proof}
\begin{corollary}\label{cor4.6}
If $G$ is a semilattice, then $\A\ltimes_\beta G$ is isomorphic to $\A$.
\end{corollary}
\begin{proof}
Let $e$ be the identity element of $G$, then $g\leq e$ for each $g\in G$. Define $\psi_1: \A\to \A\ltimes_\beta G$ by $a\mapsto \Phi(a\delta_e)$. Obviously $\psi_1$ is a $*$-homomorphism. Now, define $\psi_2: \frac{L}{N}\to \A$ by $\Phi(a\delta_g)\mapsto a$. Now, we will show that  $\psi_2$ is well-defined. If $\Phi(a\delta_e)=\Phi(b\delta_e)$, then for each covariant representation ($\pi, \nu, \h$) we have
\begin{equation*}
  \pi\times\nu(a\delta_e - b\delta_e)= \pi(a - b)=0,
\end{equation*}
so, $a - b=0$ since $\A$ has a universal representation. This shows that $\psi_1$ is well-defined  since  $\Phi(a\delta_g)=\Phi(a\delta_e)$ for each $g\in G$. Clearly, $\psi_2$ is a $*$-homomorphism that can be extended to $\A\ltimes_\beta G$. Finally, it is easy to see that $\psi_1\circ\psi_2$ and $\psi_2\circ\psi_1$ are identity maps on $\A\ltimes_\beta G$ and $\A$ respectively.
\end{proof}
\begin{proposition}\label{pro4.7}
Let $(\Pi,\h)$ be a non-degenerate representation of $\A\ltimes_\beta G$. Define a representation $\pi$ of $\A$ on $\h$ and a map $\nu:S\to B(\h)$ by \begin{equation*}
  \pi(a)=\Pi(a\delta_e),\hspace{0.5cm} \nu_s=\lim_\lambda \Pi(u_\lambda\delta_s)\rho_{s^*},
\end{equation*}
where $\{u_\lambda\}$ is an approximate identity of $E_s$, limit is the strong limit, and $\rho_{s^*}$ is the orthogonal projection onto $\pi(E_{s^*})\h$. Then $(\pi, \nu, \h)$ is a covariant representation of $(\A, G, \beta)$.
\end{proposition}
\begin{proof}
Clearly $\pi$ is a representation of $\A$ on $\h$. Now, let $\{u_\lambda \}$ be an approximate identity for $E_s$, and let $h\in\h$. We will consider two cases:\\
If $h\in\pi(E_{s^*})\h$: then there exist elements $a\in E_{s^*}$ and $h'\in \h$  such that  $h=\pi(a)h'$. So,
\begin{eqnarray*}
  \nu_s(h) &=& \lim_\lambda\Pi(u_\lambda\delta_s)(\Pi(a\delta_e)h') \\
   &=& \lim_\lambda\Pi(u_\lambda\delta_s\ast a\delta_e)h' \\
   &=&\lim_\lambda \Pi(\beta_s(\beta_{s^*}(u_{\lambda})a)\delta_s)h' \\
  &=& \Pi(\beta_s(a)\delta_s)h'.
\end{eqnarray*}
If $h\in (\pi(E_{s^*})\h)^\bot$: by  the definition we have
 \begin{eqnarray*}
   \nu_s&=& \lim_\lambda\Pi(u_\lambda\delta_s)\rho_{s^*}h=0. \\
 \end{eqnarray*}
 This show that $\nu_s$ is independent of the choice of approximate identity of $E_s$, so $\nu$ is well-defined. Now, we want to show that $\nu^*_s=\nu_{s^*}$ for $s\in S$. First we remark that for $a_s\in E_s$ we have $\Pi(a_s\delta_s)\rho_{s^*}=\rho_s\Pi(a_s\delta_s)$. Let $\{u_\lambda \}$ be an approximate identity for $E_s$. It follows that
 \begin{eqnarray*}
   (\nu_s)^* &=& \lim_\lambda(\Pi(u_\lambda\delta_s)\rho_{s^*})^* \\
    &=& \lim_\lambda\rho_{s^*}\Pi(\beta_{s^*}(u_\lambda)\delta_{s^*}) \\
    &=&\lim_\lambda \Pi(\beta_{s^*}(u_\lambda)\delta_{s^*})\rho_{s}\\
    &=&\nu_{s^*}
 \end{eqnarray*}
 since $\{\beta_{s^*}(u_\lambda)\}$ is an approximate identity for $E_{s^*}$. As for the covariance condition, let $x\in E_{s^*}$ and observe that
 \begin{eqnarray*}
   \nu_s\pi(x)\nu_{s^*} &=& \lim_{\lambda,\mu}\rho_s\Pi(u_\mu\delta_s)\Pi(x\delta_e)\Pi(\beta_{s^*}(u_\lambda)\delta_{s^*})\rho_s \\
    &=& \lim_{\mu,\lambda}\rho_s\Pi(u_\mu\delta_s\ast x\delta_e\ast \beta_{s^*}(u_\lambda)\delta_{s^*}) \rho_s\\
    &=& \lim_{\mu,\lambda}\rho_s\Pi(u_\mu\beta_s(x)u_\lambda\delta_{ss^*})\rho_s \\
    &=& \lim_{\mu,\lambda}\rho_s\Pi(u_\mu\beta_s(x)u_\lambda\delta_{e})\rho_s \\
    &=& \rho_s\pi(\beta_s(x))\rho_s \\
    &=& \pi(\beta_s(x)).
 \end{eqnarray*}
It should be noted that we have used the fact that $\Pi\equiv0$ on $N$ in the forth equality above. As for property (2) of Definition \ref{def1}, notice that for $a_s\in E_s$ we have
 \begin{eqnarray*}
   \Pi(a_s\delta_s) &=& \lim_\lambda\Pi(a_su_\lambda\delta_s) \\
    &=& \lim_\lambda\Pi(a_s\delta_e\ast u_\lambda\delta_s) \\
    &=& \pi(a_s)\rho_s\lim_\lambda\Pi(u_\lambda\delta_s) \\
    &=& \pi(a_s)\lim_\lambda\Pi(u_\lambda\delta_s)\rho_{s^*} \\
    &=& \pi(a_s)\nu_s.
 \end{eqnarray*}
 Thus,
 \begin{eqnarray*}
   \Pi(a_s\delta_s)\Pi(a_t\delta_t) &=& \pi(a_s)\nu_s\pi(a_t)\nu_t \\
    &=& \nu_s\nu_{s^*} \pi(a_s)\nu_s\pi(a_t)\nu_t \\
    &=& \nu_s\pi(\beta_{s^*}(a_s)a_t)\nu_t \\
    &=& \nu_s\pi(\beta_{s^*}(a_s)a_t)\nu_{s^*}\nu_s\nu_t \\
    &=& \pi(\beta_s(\beta_{s^*}(a_s)a_t))\nu_s\nu_t.
 \end{eqnarray*}
 Because $\Pi$ is multiplicative, the above expression is the same as
 \begin{eqnarray*}
   \Pi(a_s\delta_s\ast a_t\delta_t) &=& \Pi(\beta_s(\beta_{s^*}(a_s)a_t)\delta_{st}) \\
   &=& \pi(\beta_s(\beta_{s^*}(a_s)a_t))\nu_{st}.
 \end{eqnarray*}
 Elements of the form $\beta_{s^*}(a_s)a_t$ generate $E_{s^*}E_t$. Since $\beta_s$ maps $E_{s^*}E_t$ onto $E_sE_{st}$, it follows that elements of the form $\beta_s(\beta_{s^*}(a_s)a_t)$ generate $E_sE_{st}$ and so property (2) of Definition \ref{def1} follows. Clearly, $\pi$ is a non-degenerate representation of $\A$. Thus ($\pi, \nu, \h$) is a covariant representation of ($\A, S, \beta$).
 \end{proof}
 \begin{proposition}\label{pro4.8}
 The correspondence $(\pi, \nu , \h)\leftrightarrow (\pi\times\nu, \h)$is a bijection between covariant representations of ($\A, S, \beta$) and non-degenerate representations of $\A\ltimes_{\beta} S$.
 \end{proposition}
 \begin{proof}
 We will show that the correspondences $(\pi, \nu, \h)\mapsto(\pi\times \nu, \h)$ and $(\Pi,\h)\mapsto(\pi, \nu, \h)$ are inverses of each other. Let ($\pi', \nu', \h$) be a covariant representation of ($\A, S, \beta$). Let ($\pi, u, \h$) be a covariant representation of ($\A, S, \beta$) induced by $\pi'\times\nu'$. Then for $a\in\A$ and $s\in S$ we have
 \begin{equation*}
   \pi(a)=\pi'\times\nu'(a\delta_e)=\pi'(a)
 \end{equation*}
 and
 \begin{eqnarray*}
   u_s&=&\lim_\lambda\rho_s\pi'\times\nu'(\omega_\lambda\delta_s)\\
   &=&\lim_\lambda\rho_s\pi'(\omega_\lambda)\nu'_s\\
   &=&\lim_\lambda\pi'(\omega_\lambda)\nu'_s=\nu'_s.
 \end{eqnarray*}
 We have used the fact that $\rho_s\pi'(\omega_\lambda)\nu'_s=\pi'(\omega_\lambda)\nu'_s$ since $\rho_s$ is the orthogonal projection onto $\pi'\times\nu'(E_s)\h=\bar{\textsf{Span}}\{\pi'(a_s)\nu'_s\;:\;a_s\in E_s\}$. Let $\Pi$ be a non-degenerate representation of $\A\ltimes_\beta S$ on $\h$. Let $(\pi, \nu, \h)$ be a covariant representation of ($\A, S, \beta$) induced by $\Pi$. Then if $a_s\in E_s$ we have
 \begin{eqnarray*}
   \pi\times\nu(a_s\delta_s) &=& \pi(a_s)\nu_s \\
    &=& \Pi(a_s\delta_e)\lim_\lambda\Pi(u_\lambda\delta_s)\rho_{s^*} \\
    &=& \Pi(a_s\delta_e)\rho_s\lim_\lambda\Pi(u_\lambda\delta_s) \\
    &=& \Pi(a_s\delta_e)\lim_\lambda\Pi(u_\lambda\delta_s)\\
   &=&\lim_\lambda\Pi(a_su_\lambda\delta_s)=\Pi(a_s\delta_s).
 \end{eqnarray*}
 Thus the correspondence is bijective.
  \end{proof}
 \section{Conection Between Croossed Products}\label{five}
 Throughout this section we will assume that $G$ is an inverse semigroup with unit element $e$.
 \begin{lemma}\label{lem5.1}
 Let ($\A, G, \alpha$) and ($\A, S_G, \beta$) be as in Proposition \ref{pro3.3}. Let ($\rho, z, \K$) be a covariant representation of ($\A, S_G, \beta$), and define a covariant representation ($\rho, \omega, \K$) of ($\A, G, \alpha$) by $\omega_g=z(\alpha_g, u_g)$ as in Proposition \ref{pro3.5}. Then $(\rho\times z)(\A\ltimes_\beta S)= (\rho\times\omega)(\A\times_\alpha G)$.
 \end{lemma}
\begin{proof}
For $g\in G$, let $s=(\alpha_g, u_g)\in S$, then $E_s=D_g$, so, $\rho(D_g)\omega_g=\rho(E_s)z_s$. Thus,
\begin{equation}\label{eq}
  \sum_{g\in G} \rho(D_g)\omega_g\subseteq\sum_{s\in S}\rho(E_s)z_s.
\end{equation}
On the other hand, if 
\begin{equation*}
  s=(\alpha_{g_1}...\alpha_{g_n}, u_{g_1}...u_{g_n})\;\text{and}\;a\in E_s=D_{g_1}D_{g_1g_2}...D_{g_1...g_n}
\end{equation*}
then by Corollary \ref{cor2.10} we have
\begin{eqnarray}\label{eqn}
  \nonumber \rho(a)z_s&=&\rho(a)z_{(\alpha_{g_1}, u_{g_1})...(\alpha_{g_n}, u_{g_n})}  \\
  \nonumber &=& \rho(a)z_{(\alpha_{g_1}, u_{g_1})}...z_{(\alpha_{g_n}, u_{g_n})}\\
 &=& \rho(a)\omega_{g_1}...\omega_{g_n}.
\end{eqnarray}
Let $\Phi(\sum a_g\delta_g)\in\frac{L}{N}$. Then by \ref{eq} we have
\begin{equation*}
  \rho\times\omega(\Phi(\sum a_g\delta_g))=\sum\rho(a_g)\omega_g\subseteq (\rho\times z)(\A\ltimes_\beta S),
\end{equation*}
 so, $(\rho\times\omega)(\A\ltimes_\alpha G)\subseteq (\rho\times z)(\A\ltimes_\beta S)$. If $\Phi(\sum a_s\delta_s)\in \A\ltimes _\beta S$, then
\begin{equation*}
  \rho\times z(\Phi(\sum a_s\delta_s))=\sum \rho(a_s)z_s\in \rho\times\omega(\A\ltimes_\alpha G)
\end{equation*}
by \ref{eqn}.
\end{proof}
\begin{theorem}\label{theorem5.2}
Let $\alpha$ be a partial action of a unital inverse semigroup $G$ on a $C^*$-algebra $\A$ such that the representation $\pi\times u$ of $\A\ltimes G$ is faithful. Define an inverse semigroup $S_G$ by $S_G=\{(\alpha_{g_1}...\alpha_{g_n}, u_{g_1}...u_{g_n})\;:\;g_1,...,g_n\in G\}$ and an action $\beta$ of $S_G$ by $\beta_s=\alpha_{g_1}...\alpha_{g_n}$ for $s=(\alpha_{g_1}...\alpha_{g_n}, u_{g_1}...u_{g_n})$, as in Proposition \ref{pro3.3}. Then the crossed product $\A\ltimes_\alpha G$ and $\A\ltimes_\beta S$ are isomorphic.
\end{theorem}
\begin{proof}
Let $\nu_s=u_{g_1}...u_{g_n})$ for $s=(\alpha_{g_1}...\alpha_{g_n}, u_{g_1}...u_{g_n})$. We know from Proposition \ref{pro3.5} $(\pi, \nu, \h)$ is a covariant representation of $(\A, S, \beta)$. If we show that $\pi\times\nu$ is a faithful representation of $\A\ltimes_\beta S$, then $(\pi\times\nu)^{-1}\circ\pi\times\nu$ is an isomorphism. Consider the universal representation of $\A\ltimes_\beta S$, which by proposition \ref{pro4.8} must be in the form $\rho\times z$ for some covariant representation $(\rho, z)$ of $(\A, S, \beta)$. By Proposition \ref{pro3.5} the definition $\omega_g= z_{(\alpha_g, u_g)}$ gives a covariant representation $(\rho, \omega, \K)$ of $(\A, G, \alpha)$ and we have $(\rho\times\omega)(\A\ltimes_\alpha G)= (\rho\times z)(\A\ltimes_\beta S)$ by Lemma \ref{lem5.1}. Put $\Theta(x)=(\rho\times\omega)(\pi\times u)^{-1}(x)$, thus, $\Theta\circ\pi\times u=\rho\times\omega$. We will show that $\Theta\circ(\pi\times\nu)=\rho\times z$. It suffices to check this on generators $a\delta_s$, where $s=(\alpha_{g_1}...\alpha_{g_n}, u_{g_1}...u_{g_n})$ and $a\in E_s=D_{g_1}...D_{g_1...g_n}$.
\begin{eqnarray*}
  \Theta((\pi\times\nu)(a\delta_s)) &=& \Theta(\pi(a)\nu_s) \\
   &=& (\rho\times\omega)(\pi\times u)^{-1}(\pi(a)\nu_s)\\
   &=& (\rho\times\omega)(\pi\times u)^{-1}(\pi(a)u_{g_1}...u_{g_n}) \\
   &=&  (\rho\times\omega)(\pi\times u)^{-1}(\pi(a)u_{g_1...g_n})\\
   &=& (\rho\times\omega)(a\delta_{g_1...g_n})\\
   &=&  \rho(a)\omega_{g_1...g_n}\\
   &=&  \rho(a)\omega_{g_1}...\omega_{g_n}\\
   &=& \rho(a)z_{(\alpha_{g_1},u_{g_1})}...z_{(\alpha_{g_n}, u_{g_n})}\\
   &=&\rho(a)z_{(\alpha_{g_1}...\alpha_{g_n},u_{g_1}... u_{g_n})}\\
   &=&\rho\times z (a\delta_s)
\end{eqnarray*}
where we have appealed to Corollary \ref{cor2.10} twice more.
\end{proof}

\end{document}